\documentclass[12pt]{article}
\usepackage{}

\usepackage{epsfig}
\usepackage{latexsym}
\usepackage{caption}
\usepackage{amsfonts}
\usepackage{amssymb}
\usepackage{mathrsfs}
\usepackage{amsmath}
\usepackage{enumerate}
\usepackage{graphicx,graphics}
\usepackage{MnSymbol}
\usepackage{float}
\usepackage{pict2e}
\usepackage{subfigure}

\usepackage{color}

\makeatletter

\newcommand{\Rmnum}[1]{\expandafter\@slowromancap\romannumeral #1@}

\makeatother

\begin{document}

\newtheorem{theorem}{Theorem}
\newtheorem{observation}[theorem]{Observation}
\newtheorem{corollary}[theorem]{Corollary}
\newtheorem{algorithm}[theorem]{Algorithm}
\newtheorem{definition}[theorem]{Definition}
\newtheorem{guess}[theorem]{Conjecture}
\newtheorem{claim}[theorem]{Claim}
\newtheorem{problem}[theorem]{Problem}
\newtheorem{question}[theorem]{Question}
\newtheorem{lemma}[theorem]{Lemma}
\newtheorem{proposition}[theorem]{Proposition}
\newtheorem{fact}[theorem]{Fact}

\makeatletter
  \newcommand\figcaption{\def\@captype{figure}\caption}
  \newcommand\tabcaption{\def\@captype{table}\caption}
\makeatother

\newtheorem{acknowledgement}[theorem]{Acknowledgement}

\newtheorem{axiom}[theorem]{Axiom}
\newtheorem{case}[theorem]{Case}
\newtheorem{conclusion}[theorem]{Conclusion}
\newtheorem{condition}[theorem]{Condition}
\newtheorem{conjecture}[theorem]{Conjecture}
\newtheorem{criterion}[theorem]{Criterion}
\newtheorem{example}[theorem]{Example}
\newtheorem{exercise}[theorem]{Exercise}
\newtheorem{notation}{Notation}
\newtheorem{solution}[theorem]{Solution}
\newtheorem{summary}[theorem]{Summary}

\newenvironment{proof}{\noindent {\bf
Proof.}}{\rule{3mm}{3mm}\par\medskip}
\newcommand{\remark}{\medskip\par\noindent {\bf Remark.~~}}
\newcommand{\pp}{{\it p.}}
\newcommand{\de}{\em}
\newcommand{\mad}{\rm mad}
\newcommand{\qf}{Q({\cal F},s)}
\newcommand{\qff}{Q({\cal F}',s)}
\newcommand{\qfff}{Q({\cal F}'',s)}
\newcommand{\f}{{\cal F}}
\newcommand{\ff}{{\cal F}'}
\newcommand{\fff}{{\cal F}''}
\newcommand{\fs}{{\cal F},s}
\newcommand{\s}{\mathcal{S}}
\newcommand{\G}{\Gamma}
\newcommand{\g}{(G_3, L_{f_3})}
\newcommand{\wrt}{with respect to }
\newcommand {\nk}{ Nim$_{\rm{k}} $  }
\newcommand {\dom}{ {\rm Dom}  }
 \newcommand {\ran}{ {\rm Ran}  }

\newcommand{\ch}{{\rm ch}}

\newcommand{\q}{\uppercase\expandafter{\romannumeral1}}
\newcommand{\qq}{\uppercase\expandafter{\romannumeral2}}
\newcommand{\qqq}{\uppercase\expandafter{\romannumeral3}}
\newcommand{\qqqq}{\uppercase\expandafter{\romannumeral4}}
\newcommand{\qqqqq}{\uppercase\expandafter{\romannumeral5}}
\newcommand{\qqqqqq}{\uppercase\expandafter{\romannumeral6}}

\newcommand{\qed}{\hfill\rule{0.5em}{0.809em}}

\newcommand{\var}{\vartriangle}

\title{{\large \bf Colouring of $S$-labeled planar graphs }}

\author{Ligang Jin\thanks{Department of Mathematics, Zhejiang Normal University,  China. Email: ligang.jin@zjnu.cn. Grant number: NSFC 11801522 and QJD1803023} \and Tsai-Lien Wong\thanks{Department of Applied Mathematics, National Sun Yat-sun University, Kaohsiung, Taiwan. Email: tsailienwong@gmail.com. Grant number: 106-2115-M-110-001-MY2} \and Xuding Zhu\thanks{Department of Mathematics, Zhejiang Normal University,  China.  E-mail: xudingzhu@gmail.com. Grant Numbers: NSFC 11571319 and 111 project of Ministry of Education of China.}}

\maketitle

\begin{abstract}
	Assume $G$ is a graph and $S$ is a set of permutations of integers.
	An $S$-labeling of $G$ is a pair $(D,\sigma)$, where $D$ is an orientation of $G$
	and $\sigma: E(D) \to S$ is a mapping which assigns to each arc $e=(u,v)$ of $D$ a permutation
	$\sigma_e \in S$. 
	A proper $k$-colouring of $(D,\sigma)$ is a mapping 
	  $f: V(G) \to [k]=\{1,2, \ldots, k\}$  	such that  $\sigma_e(f(x)) \ne f(y)$ for each  arc $e=(x,y)$. 
	We say $G$ is   $S$-$k$-colourable if  any $S$-labeling $(D, \sigma)$ of $G$ has a proper $k$-colouring.
	The concept of $S$-$k$-colouring is a common generalization of many colouring concepts, including   $k$-colouring,  signed $k$-colouring,  signed   $Z_k$-colouring,  
	  DP-$k$-colouring,  group colouring and colouring of gained graphs. 
 We are interested in the problem
as for which subset $S$ of $S_4$, every planar graph is $S$-$4$-colourable.
We call such  a subset $S$  a good subset. The famous Four Colour Theorem 
is equivalent to say  that   $S=\{id\}$ is good. A result of  Kr\'{a}l,   
Pangr\'{a}c and Voss is equivalent to say that $S=\{id, (1234),(13)(24), (1432)\}$
 and $S=\{id, (12)(34),(13)(24), (14)(23)\}$ are not good.  
 These results are strengthened by a 
  very recent result of  Narboni and Tarkos, which implies that 
  $S=\{id, (12)(34)\}$ is not good and another  very recent result of
   Zhu  which implies that   $S=\{id, (12)\}$ is not good.  This paper proves 
   if $S$ is a susbet of $S_4$ containing $id$, then $S$ is good if and only if $S=\{id\}$.

\noindent {\bf Keywords:}
signed graph colouring, gained graphs, $S$-labeled graph, DP-colouring,    group colouring.

\end{abstract}


\section{Introduction}

A {\em signed graph } is a pair $(G, \sigma)$, where $G$ is a graph and $\sigma: E(G) \to \{1,-1\}$ assigns to each edge $e$ a sign $\sigma_e \in \{1,-1\}$. Colouring of signed graphs has been studied in many papers. There are a few different definitions of colouring of signed graphs, all of them are natural generalizations of colouring of (unsigned) graphs. 

In
the
1980's,
Zaslavsky  
 \cite{Z} defined a  colouring of a signed graph $(G, \sigma)$ as a mapping $f: V(G) \to \{\pm k, \pm(k-1) \ldots, \pm 1, 0\}$ such that for any edge $e=xy$ of $G$, $f(x) \ne \sigma_e f(y)$. He distinguished between colourings of $(G, \sigma)$ using colour $0$ and that not using colour $0$. 
 In 2016, M\'{a}\v{c}ajov\'{a},   Raspaud and \v{S}koviera \cite{MRS}  modified the definition of $k$-colouring of 
 a  signed graph   as follows:
\begin{definition}
		\label{def1}
	Assume $(G,\sigma)$ is a signed graph and $k$ is a positive integer.
	Let 
	\[
	N_k=\begin{cases} \{\pm q, \pm(q-1) \ldots, \pm 1 \}, &\text{ if $k=2q$ is even}, \cr
	\{\pm q, \pm(q-1) \ldots, \pm 1,0 \}, &\text{ if $k=2q+1$ is odd}.
	\end{cases}
	\]
	A {\em proper $k$-colouring} of $(G, \sigma)$ is a mapping 
	$f: V(G) \to N_k$ such that  for any edge $e=xy$ of $G$, $f(x) \ne \sigma_e f(y)$. 
	We say $G$ is {\em signed $k$-colourbale} if for any signature $\sigma$ of $G$, 
	$(G, \sigma)$ has a proper $k$-colouring.  
\end{definition}

In 2017, Kang and Steffen introduced another type of  colouring of signed graphs \cite{KS2, KS}, which we call $Z_k$-colouring.
 
 \begin{definition}
 	\label{def2}
 	A  \emph{proper $Z_k$-colouring} of a signed graph $(G, \sigma)$ is a mapping $f: V(G) \to Z_k$ such that for each   edge $e=uv$, $f(u) \ne \sigma_e f(v)$.  Here   $-f(v)$ is the inverse of $f(v)$ in $Z_k$, i.e., it is equal to $-c(v) \pmod{k}$.  
 	We say a graph $G$ is  \emph{signed $Z_k$-colourable}  if for any signature $\sigma$ of $G$, $(G, \sigma)$ is has a proper  $Z_k$-colouring.
 \end{definition} 	
 
For odd integer $k$, a graph $G$ is signed $k$-olourable if and only if $G$ is signed $Z_k$-colourable.
However if $k$ is even, then there are graphs $G$ that are signed $k$-colourable but not signed $Z_k$-colourable and
also graphs that are signed $Z_k$-colourable but not signed $k$-colourable.
 
Another colouring of signed graphs adopts the view of graph homomorphisms.
The chromatic number of a graph $G$ can be defined as the minimum order of a graph $H$ such that $G$ admits a homomorphism to $H$.  Homomorphisms of signed graphs was studied in \cite{Reza}. For a signed graph
$(G, \sigma)$, a {\em switching} at vertex $v$ results another signed graph $(G, \sigma')$, where 
$\sigma'_e = -\sigma_e$ for all edges $e$ incident to $v$, and 

For two signed graph $(G, \sigma)$ and $(H, \tau)$, we write $(G, \sigma) \to (H, \tau)$, if there is a signed graph $(G, \sigma')$ obtained from $(G, \sigma)$ by switching at some vertices that admits a homomorphism to $(H, \tau)$ (i.e., a mapping from $V(G) \to V(H)$ that preserves the adjacency and also the sign the edges). Then the \emph{chromatic number} of a signed graph $(G, \sigma)$ is   defined in \cite{Reza} to be the minimum order of a signed graph $(H, \tau)$ such that $(G, \sigma) \to (H, \tau)$.

As generalizations of the Four Colour Theorem, M\'{a}\v{c}ajov\'{a},   Raspaud and \v{S}koviera \cite{MRS} conjectured that every planar graph is signed $4$-colourable, and Kang and Steffen \cite{St} conjectured that 
every  planar graph is signed  $Z_4$-colourable.
Very recently (one year after the submission of this paper), both conjectures are disproved.  Narboni and Tarkos   \cite{Tarkos} constructed a non-$4$-colourable signed planar graph, 
and Zhu \cite{Zhurefined} constructed a non-$Z_4$-colourable signed planar graph. 

In a colouring of a signed graph $(G, \sigma)$, the sign $\sigma_e$ of an edge $e$ 
is used to impose restriction on the colour pairs that can be assigned to the end vertices of $e$. 
There are two
possible signs for the edges, and hence there are two kinds of restrictions for the pairs of colours that can be assigned to the end vertices of an edge.   

It is natural to consider more  types of restrictions for pairs of colours that can
be assigned to the end vertices of an edge. Indeed, various colourings of graphs corresponding to
different restrictions have been studied in the literature.
 Below are three types of graph colourings that are examples of such constraints. 

In 1992,
Jaeger,  Linial, Payan  and   Tarsi \cite{JLPT1992} introduced the concept of group colouring.
Assume $\Gamma$ is an Abelian group, $D$ is an orientation of $G$ and $\sigma: E(D) \to \Gamma$ assigns to each arc $e=(u,v)$ an element $\sigma(e)$. A $\Gamma$-colouring of $(D, \sigma)$ is a mapping $f: V(G) \to \Gamma$ such that for any arc $e=(u,v)$ of $D$, $f(v) - f(u) \ne \sigma(e)$. We say $G$ is {\em $\Gamma$-colourable } if for any orientation $D$ of $G$ and any $\sigma:E(D) \to \Gamma$, there exists a $\Gamma$-colouring of $(D, \sigma)$.
The {\em group chromatic number} of $G$ is the minimum integer $k$ such that for any Abelian group $\Gamma$ 
of order $k$, $G$ is $\Gamma$-colourable  (cf. \cite{survey} for a survey on this subject).

In 1995, as a generalization of colouring of signed graphs, Zaslavsky \cite{Z2} defined the colouring of gain graphs as follows: A {\em gain graph} consists of a graph $G$, a {\em gain group}  $\Gamma$, and a {\em gain function} $\phi$, which assigns each orientation $e=(u,v)$ of an edge $uv$  a group element  $\phi(e)$ so that $\phi(e^{-1}) = \phi(e)^{-1}$, where $e^{-1} = (v,u)$ is the inverse of $e$. We denote a gain graph as 
$(G, \phi)$. A proper  $k$-colouring of $(G,\phi)$ assigns to each vertex $v$ of $G$ a colour $f(v)$ from the colour set $\{0\} \cup \{(i,\pi): i \in \{1,2,\ldots,k\}, \pi \in \Gamma\}$,  so that for each orientation $e=(u,v)$ of an edge $uv$ of $G$,  if $f(u)=0$ then $f(v) \ne 0$, and if $f(u)=(i,\pi)$, then $f(v) \ne (i, \pi \circ \phi(e))$ (here $\circ$ is the   product in $\Gamma$).  Note that if $\Gamma$ is an Abelian group, then 
a proper $1$-colouring of  $(G,\phi)$ is the same as the $\Gamma$-colouring of $(G, \phi)$ defined as above by 
Jaeger,  Linial, Payan  and   Tarsi \cite{JLPT1992}.

In 2018, Dvo\v{r}\'{a}k and Postle \cite{DP} defined DP-colouring of a graph as follows:
 Assume $G$ is a graph and $\sigma$ is a mapping that assigns to each orientation $e= (u,v)$ of an edge $uv$ a permutation $\sigma_e$ of integers  so that $\phi(e^{-1}) = \phi(e)^{-1}$. A proper $k$-colouring of $(G, \sigma)$ is a mapping $f: V(G) \to \{1,2,\ldots,k\}$ so that for each oriented edge $e=(u,v)$,
$f(u) \ne \sigma_e(f(v))$. We say $G$ is DP-$k$-colourable if for any $\sigma$, there exists a proper $k$-colouring of $(G, \sigma)$. (The definition of DP-$k$-colouring given above is different but equivalent to the definition given 
in \cite{DP}).  It was shown in \cite{DP} that every DP-$k$-colourable graph is $k$-choosable, and the concept of DP-colouring is used to show that every planar graph without cycles of lengths $4,5,6,7,8$ is $3$-choosable, which is a problem remained open for 15 years. 

In each of the colourings defined above, edges of $G$ are labeled, and certain pairs of colours are forbidden to be assigned to the end vertices of edges with given label. In this sense, the   labels play the same role as the sign
of the edges in a signed graph. However, instead of $+$ and $-$ signs, we may have many different signs. Adopting this point of view, we introduce the concept of colouring of   {\em $S$-labeled graphs } as follows.

\begin{definition}
	Assume $S$ is  a set  of permutations of 
	positive integers.
	An {\em $S$-labeling} of $G$ is an orientation $D$ of $G$ together with a mapping $\sigma: E(D) \to S$.
	The pair $(D, \sigma)$ is called an {\em $S$-labeled graph}.
	A {\em proper $k$-colouring } of $(D, \sigma)$ is a mapping $f: V(D) \to [k]=\{1,2,\ldots, k\}$ such that for each 
	arc $e=(x,y)$ of $D$, 
	$\sigma_e(f(x)) \ne f(y)$.
	A graph $G$ is called {\rm $S$-$k$-colourable } if $(D, \sigma)$ is  $k$-colourable for every $S$-labeling $(D, \sigma)$ of $G$.
\end{definition}

We denote by $id$ the identity permutation on the set $[k]=\{1,2,\ldots,k\}$,
 and denote by $S_k$ the symmetric group of order $k$.
The following follows easily from the definition.
 \begin{enumerate}
 	\item If $S=\{id\}$, then $S$-$k$-colourability is equivalent to $k$-colourability.
 	\item If $S=\{id, (12)(34)\ldots ((2q-1)(2q))\}$ where $q = \lfloor k/2 \rfloor$, then
   $S$-$k$-colourability is equivalent to signed   $k$-colourability.
   \item  If $S=\{id, (12)(34)\ldots ((2q-1)(2q))\}$ where $q = \lceil k/2 \rceil -1$,    then
   $S$-$k$-colourability is equivalent to signed   $Z_k$-colourability.
   \item If $S=S_k$, then $S$-$k$-colourability 
   is equivalent to DP-$k$-colourability.
   \item If $S=Z_k$, then $S$-$k$-colourability is the same as   group $Z_k$-colourability.
   \item Assume $\Gamma$ is a  group with $|\Gamma|=n$ and $k$ is a positive integer.
   Let $\tau: \Gamma \to [n]$ be a one-to-one correspondence from $\Gamma $ to $[n]$.
   Let $k'=kn+1$ and for each $\pi \in \Gamma$, let 
   $\pi'$ be the permutation of $[k']$ defined as 
   $\pi'(nj+r) = nj+ \tau( \tau^{-1}(r) \circ \pi) $ for $r \in [n]$ and $j \in \{0,1,\ldots, k-1\}$, and $\pi'(kn+1)=kn+1$. Let $S=\{\pi': \pi \in \Gamma\}$.
   Then a $k$-colouring of a gain graph $(G, \phi)$  with gain group $\Gamma$ is equivalent 
   to a $k'$-colouring of a $S$-labeling $(D, \sigma)$ of $G$, where $\sigma$ is defined as 
   $\sigma(e)=\pi'$ if $\phi(e)=\pi$.
 \end{enumerate}

In this paper, we are interested in $S$-$4$-colouring of planar graphs. 

By Four Colour Theorem, if $S=\{id\}$, then every planar graph is $S$-$4$-colourable. The existence of
 non-$4$-choosable planar graphs shows that, if $S$ is the set of all permutations of integers,  not every planar graph is $S $-$4$-colourable. Note that for if $S$ of permutations of integers, $\pi \in S$ and integer $i$, if $i \notin \{1,2,\ldots, k\}$ or 
  $\pi(i) \notin \{1,2,\ldots, k\}$, then $\pi$   puts no restriction  on vertices of colour $i$ or $\pi(i)$. For this reason, we may restriction our attention to permutations of $\{1,2,\ldots, k\}$ when we consider $S$-$k$-colourings of graphs.
 An interesting problem is for which subsets $S$ of $S_4$, every planar graph is $S$-$4$-colourable.

 \begin{definition}
 \label{def-goodbad}
 Assume $S$ is a  non-empty subset of $S_4$. We say $S$ is {\em good} if every planar graph is $S$-$4$-colourable, and $S$ is {\em bad} otherwise.
 \end{definition}

The Four Colour Theorem is equivalent to say that $S=\{id\}$ is good. We are interested in the problem as which subsets $S$ of $S_4$ are good. In this paper, we consider subsets $S$ containing $id$. In some sense, we are interested in the problem as how tight is the Four Colour Theorem. If any subset $S$ of $S_4$ containing $id$ but other than $\{id\}$ is good, then that would be a strengthening of the Four Colour Theorem. As we mentioned earlier, the study of colouring of signed graphs motivated two conjectures of colouring of signed planar graphs that are strengthening of the Four Colour Theorem of this type.

Observe that if $S, S'$ are   subsets of $S_k$ and $S \subseteq S'$, then any $S'$-$k$-colourable graph is $S$-$k$-colourable. We say $S'$ is a conjugation of $S$ if there is a permutation $\pi \in S_k$ such that 
$S' = \{\pi \sigma \pi^{-1}: \sigma \in S\}$. It is easy to see that if $S'$ is a conjugation of $S$, then 
a graph $G$ is $S'$-$k$-colourable if and only if $G$ is $S$-$k$-colourable.

It was proved by 
  Kr\'{a}l,   
  Pangr\'{a}c and Voss \cite{Kral} that there are planar graphs that are not $Z_4$-colourable and there are planar graphs that are not $Z_2 \times Z_2$-colourable. This means that  $S=\{id, (1234),(13)(24), (1432)\}$
  and $S=\{id, (12)(34),(13)(24), (14)(23)\}$ are not good.  
  In this first submitted version of this paper, we showed that the following subsets of $S_4$ are bad:
  $$\{id, (123) \}, \{id, (1234) \}, \{(id, (12), (13)\}, \{id, (12)(34), (13)\}, \{id, (12)(34), (13)(24)\}.$$ 
  
  As mentioned above,  Narboni and Tarkos \cite{Tarkos} refuted the conjecture of M\'{a}\v{c}ajov\'{a}, Raspaud and Skoviera, and showed that there are planar graphs that are not signed $4$-colourable. This means that $S=\{id, (12)(34)\}$ is bad.
  Zhu \cite{Zhurefined} refuted the conjecture of Kang and Steffen, and showed that there are planar graphs that are not signed 
  $Z_4$-colourable. This means that $\{id, (12)\} $ is bad. These results covered some of the cases listed above.
   The cases uncovered by these two results are   $S=\{id, (123)\}$ and $S=\{id, (1234)\}$. In this paper, we show that $S=\{id, (123)\}$ and $S=\{id, (1234)\}$ are bad. Therefore a subset $S$ of $S_4$ containing $id$ 
   is good if and only if   $S = \{id\}$.

\section{The Proof}

In this section, for $S=\{id, (123)\}$ or $S=\{id, (1234)\}$, we construct planar graphs $G$ and 
$S$-labelings $(D, \sigma)$  of $G$ that are not $S$-4-colourable.

\begin{definition}
	A graph is uniquely $k$-colourable if there is a unique partition of $V(G)$ into $k$ independent sets.
\end{definition}

Assume $G$ is uniquely $k$-colourable, and $V_1, V_2, \ldots, V_k$ is the unique partition of $V(G)$ into $k$ independent sets. There are $k!$ ways of assigning the $k$ colours $\{1,2,\ldots, k\}$ to the independent sets. So there are actually $k!$ $k$-colourings of $G$. If $G$ is a uniquely $4$-colourable planar graph, then there are exactly $24$ $4$-colourings of $G$.

For a plane graph $G$, we denote by ${\cal F}(G)$ the set of faces of $G$.

\begin{lemma}
	\label{lemma-unique}
	There exists a uniquely $4$-colourable plane triangulation $G'$, a set ${\cal F}$ of $24$ faces of $G'$ and a one-to-one correspondence $\phi$ between ${\cal F}$ and the $24$ $4$-colourings of $G'$ such that for each $F \in {\cal F}$, $\phi_F(V(F)) = \{1,2,3\}$, where $\phi_F$ is the $4$-colouring of $G'$ corresponding to $F$ and $V(F)$ is the set of vertices incident to $F$.
\end{lemma}
\begin{proof}
	Build a plane triangulation which is uniquely $4$-colourable and which has $24$ faces. Such a graph can be constructed by starting  from a triangle $T=uvw$, and repeat the following: choose a face $F$
	(which is a triangle), add a vertex $x$ in the interior of $F$ and connect $x$ to each of the three vertices of $F$. Each iteration
	of this procedure increases the number of faces of $G$ by $2$. We stop when there are $24$ faces.
	
	Let $\phi$ be an arbitrary one-to-one correspondence between the $24$   $4$-colourings of $G$ and the $24$ faces of $G$.
	For each face $F$ of $G$, we denote by $\phi_F$ the corresponding $4$-colouring of $G$.
	
	Let ${\cal F}'$ be the set of faces $F$ for which $\phi_F(V(F)) \ne \{1,2,3\}$.
	
	For each $F \in {\cal F}'$, add a vertex $z_F$ in the interior of $F$, connect $z_F$ to each of the three vertices of $F$. The colouring $\phi_F$ is uniquely extended to $z_F$.   Hence the resulting plane triangulation $G'$ is still uniquely $4$-colourable.
	The face $F$ of $G$ is partitioned into three faces of $G'$. One of the three faces  is coloured by $\{1,2,3\}$. We denote this face by $F'$ and  use this face of $G'$ instead of the face $F$ of $G$ to be associated with   the colouring $\phi_F$ (and we denote this colouring by $\phi_{F'}$ after this operation).
	
	Let $${\cal F}=\{F': F \in {\cal F}'\} \cup ({\cal F}(G)-{\cal F}').$$
	
	The   one-to-one correspondence   between ${\cal F}$ and the $24$ $4$-colourings of $G'$ defined above satisfies the requirements of the lemma.
\end{proof}

\begin{theorem}
\label{main}
Assume $S$ is a subset of $S_4$ which contains $id$. If $S \ne \{id\}$, then   $S$ is bad.
\end{theorem}
\begin{proof}
	Assume $S$ is a subset of $S_4$ which contains $id$. Assume $\pi \in S - \{id\}$.
	If $\pi=(12)$ or $\pi=(12)(34)$, then it follows from the above mentioned results of    
	  Narboni and Tarkos \cite{Tarkos}, and the result of Zhu \cite{Zhurefined} that $S$ is bad.
	 Thus we may assume that $\pi=(123)$ or $\pi=(1234)$.
	
By Lemma \ref{lemma-unique}, there is a    uniquely $4$-colourable plane triangulation $G'$,    a set ${\cal F}$ of $24$ faces of $G'$ and    a one-to-one correspondence $\phi$ between ${\cal F}$ and the $24$ $4$-colourings of $G'$ so that for each $F \in {\cal F}$, $\phi_F(V(F))=\{1,2,3\}$.

For each face $F \in {\cal F}$,  for  $i \in \{1,2,3\}$, let $v_{F,i}$ be the vertex with $\phi_F(v_{F,i})=i$.
Note that a vertex $v$ may belong to distinct faces $F,F' \in {\cal F}$. If $\phi_F(v)=i$ and $\phi_{F'}(v)=j$, then 
$v=v_{F,i} = v_{F', j}$. 
 
\begin{itemize}
\item   add a triangle $T_F = a_Fb_Fc_F$ in the interior of $F$;
\item   connect $a_F$ to   $v_{F,1}$ and $v_{F,2}$; connect $b_F$ to   $v_{F,1}$ and $v_{F,3}$; connect $c_F$ to $v_{F,2}$ and $v_{F,3}$.
\end{itemize}

We denote the resulting plane triangulation by $G$. 

\bigskip

\noindent 1:    If $S=\{id, (123)\}$, then $S$ is bad.

Let $D$ be the orientation of $G$ in which  the edges of $T_F$ are oriented as $(b_F,a_F)$, $(c_F,b_F)$, $(a_F,c_F)$.
The orientation of the other edges are arbitrary (and irrelevant).
See Figure 1.

\begin{figure}[h]
	\label{fig_123}
	\begin{center}
		\includegraphics[scale=0.7]{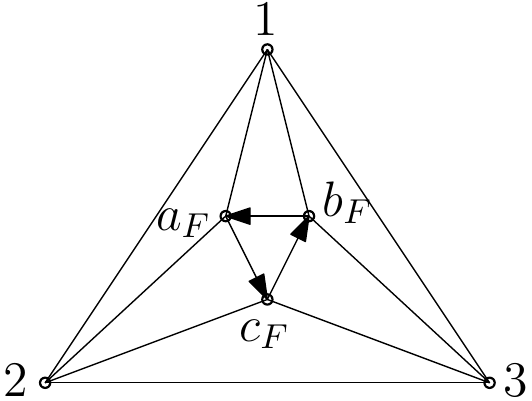} 
	\end{center}
	\caption{$S=\{id, (123) \}$ is bad, where arrowed edges $e$ have $\sigma_e=(123)$}
\end{figure}

 Let $\sigma: E(D) \to S$ be defined as $\sigma_e=id$ for all edges $e$, except that for each $F \in {\cal F}$, for the three edges $e$ of $T_F$, $\sigma_e=(123)$.

Now we show that $G$ is not $\sigma$-colourable.

Assume $\psi$ is an $\sigma$-colouring of $G$. Then the restriction of $\psi$ to $G'$ is a proper 4-colouring of $G'$. As $G'$ is uniquely 4-colourable, the restriction of $\psi$ to $G'$ equals $\phi_F$ for some $F \in {\cal F}$.
Consider the triangle $T_F$. The vertex $a_F$ is adjacent to vertices of colours 1 and 2 by edges $e$ with $\sigma_e=id$. Hence, $\psi(a_F) \neq 1,2.$ So $\phi(a_F) \in \{3,4\}$. Similarly, $  \psi(b_F) \in \{2,4\}$ and $\psi(c_F) \in \{1,4\}$.

 If $\psi(a_F)=4$, then since $\sigma_{(b_F, a_F)}(4) = 4$, we conclude that $\psi(b_F) \ne 4$, hence $\psi(b_F)=2$.
 Similarly, we must have $\psi(c_F)=1$.
 As $\sigma_{(c_F, b_F)}(1) =2$,    this is a contradiction.

  Assume $\psi(a_F)=3$. Then the same argument shows that 
   $\psi(c_F) =  \psi(b_F) = 4$. As $\sigma_{(c_F,b_F)} (4)=4$, this is a contradiction.

 \bigskip
 
 \noindent 2: If $S=\{id, (1234) \}$, then $S$ is bad.

 Let $D$ be the orientation of $G$ in which   the   edges incident to $a_F,b_F,c_F$ are oriented   as $(c_F,a_F), (c_F,b_F), (v_{F,3}, c_F)$ and $(v_{F,2},a_F)$, and 
   the other edges are oriented arbitrary.
  See Figure 2.

\begin{figure}[h]
	\label{fig-1234}
	\begin{center}
		\includegraphics[scale=0.7]{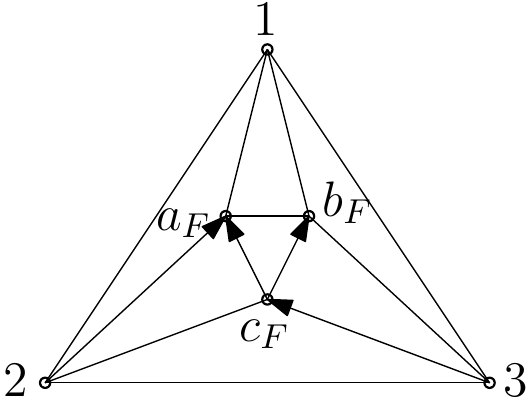} 
	\end{center}
	\caption{$S=\{id, (1234) \}$ is bad, where arrowed edges $e$ have $\sigma_e=(1234)$ }
\end{figure}

 Let $\sigma: E(D) \to S$ be defined as $\sigma_e=id$ for all edges $e$, except that for each $F \in {\cal F}$,  for $e \in \{(v_{F,2},a_F), (c_F,a_F), (c_F,b_F), (v_{F,3},c_F)\}$,  $\sigma_e=(1234)$.

 Now we show that $G$ is not $\sigma$-colourable.

 Assume $\psi$ is an $\sigma$-colouring of $G$. Then the restriction of $\psi$ to $G'$  
 equals $\phi_F$ for some $F \in {\cal F}$.  The same argument as before shows that $\psi(a_F) \in \{2,4\}$,
 $\psi(b_F) \in \{2,4\}$ and $\psi(c_F) \in \{1,3\}$.

 If $\psi(c_F)=1$, then since $\sigma_{(c_F, a_F)}(1) = 2$, we conclude that $\psi(a_F) \ne 2$, hence $\psi(a_F)=4$.
 Similarly, we must have $\psi(b_F)=4$.
 But $\sigma_{ (a_F,b_F)  }(4) =4$,    this is a contradiction.

 If $\psi(c_F)=3$, then since $\sigma_{(c_F, a_F)}(3) = 4$, we conclude that $\psi(a_F) \ne 4$, hence $\psi(a_F)=2$.
 Similarly, we must have $\psi(b_F)=2$.
 Again  this is a contradiction.

 This completes the proof of Theorem \ref{main}.
 \end{proof}


\end{document}